\documentclass[11pt]{amsart}

\usepackage{amsthm,amssymb, amsmath}

\usepackage[all]{xy}

\usepackage{graphicx}

\usepackage{subfig}
\usepackage{fullpage}

\newtheorem{theorem}{Theorem}%[section]

\newtheorem{definition}{Definition}[section]

\newtheorem{lemma}[theorem]{Lemma}

\newtheorem{proposition}[theorem]{Proposition}

\newtheorem{corollary}[theorem]{Corollary}

\title{The local H\"older exponent for the entropy of real unimodal maps}

\author{Giulio Tiozzo}

\begin{document}

%\begin{abstract}
%We prove that the topological entropy of real unimodal maps depends 
%as a H\"older continuous function of the kneading parameter, 
%and the local H\"older exponent equals, up to a factor $\log 2$, the value of the function at that point. 
%\end{abstract}

\maketitle

\begin{flushright}
\textit{For Tan Lei}
\end{flushright}
\vskip 1 cm

In order to encode and classify the topological dynamics of interval maps, Milnor and Thurston \cite{MT} defined 
the \emph{kneading sequence} of a unimodal map $f$ by recording the relative position of the iterates of the critical point. 
%on which side of the critical point the iterates of the critical point lie. 
This information can be packaged in a binary number, known as the \emph{kneading parameter}, or \emph{kneading angle} $\theta$ 
(see Section \ref{S:back} for a precise definition).  
One can check that the topological entropy of $f$ only depends on the kneading angle $\theta$, hence we can define 
$$h(\theta) := h(f)$$
the topological entropy of any unimodal map $f$ which has angle $\theta$. 

The function $h(\theta)$ has also the following interpretation in complex dynamics. Let $\theta \in \mathbb{R}/\mathbb{Z}$, and 
suppose that the external ray of angle $\theta$ for the Mandelbrot set 
%and suppose that the corresponding external ray 
lands on some real parameter $c_\theta$. Then $h(\theta)$ equals the topological entropy 
of the quadratic polynomial $f(z) = z^2 + c_\theta$. 

\medskip 
\noindent %Many things are known about the behaviour of the entropy function
The entropy of unimodal maps has been explored for several decades. In particular, the function $h(\theta)$:
\begin{enumerate}
\item is a continuous, weakly increasing function of $\theta$ \cite{MT}, and its maximum value is $\log 2$;
\item is constant on small copies of the Mandelbrot set whose root has positive entropy \cite{Do}; 
\item is positive for all $\theta >  \theta_\star$, where $\theta_\star$ is the kneading angle of the \emph{Feigenbaum map}, 
whose binary expansion is the \emph{Thue-Morse} sequence.
\end{enumerate} 

%Moreover, it is known that $h(\theta)  > 0$ if and only if $\theta > \theta_0$, where 
%Let $\theta_0$ be the largest angle for which $h(\theta_0) = 0$. It is known that 
%$\theta_0$ is the kneading angle of the \emph{Feigenbaum map} and its binary expansion is the Thue-Morse sequence.
\medskip

In this note, we are interested in the regularity of $h$. Our main result states that the local H\"older exponent of the entropy function $h(\theta)$ equals, up to a constant factor
of $\log 2$, the value of the function itself. Let us denote as $\alpha(f, z)$ the local H\"older exponent of the function $f$ at $z$ (see Section \ref{S:back} for a precise definition). It is well-known that not all binary sequences are indeed kneading sequences of unimodal maps: let us denote as $\mathcal{R}$  the set of all possible kneading angles. This coincides (up to possibly a set of Hausdorff dimension zero) with the set of angles of 
external rays which land on the real slice of the Mandelbrot set. Each connected component $U$ of the complement $[0, 1/2] \setminus \mathcal{R}$ corresponds to a real 
hyperbolic component, and it is well-known that the entropy at the two endpoints of $U$ is the same, hence one can extend the definition of entropy $h(\theta)$ to each value $\theta \in [0, 1/2]$ by setting it constant over each component $U$. 
Moreover, we define a plateau for the entropy function as an open interval in parameter space on which the entropy is constant. 

%We call an angle \emph{satellite} if the corresponding ray lands at the root of a satellite component of the Mandelbrot set. 

\begin{theorem}\label{T:main} 
For any $\theta \in \mathcal{R}$ with $\theta > \theta_\star$, the entropy function $h$ is locally H\"older continuous at $\theta$ with 
exponent $\frac{h(\theta)}{\log 2}$.
Moreover,
%\begin{itemize}
%\item[(i)] if $\theta \in \mathcal{R}$ is positively renormalizable, then $h$ is locally constant at $\theta$; 
%\item[(ii)] 
if $\theta \in \mathcal{R}$ does not lie in a plateau, then the local H\"older exponent of the entropy function at 
$\theta$ is related to the value of the function by the equation
$$\alpha(h, \theta) = \frac{h(\theta)}{\log 2}.$$
%\end{itemize}
\end{theorem}

\begin{figure} \label{F:entro}
\fbox{\includegraphics[width= 0.75 \textwidth]{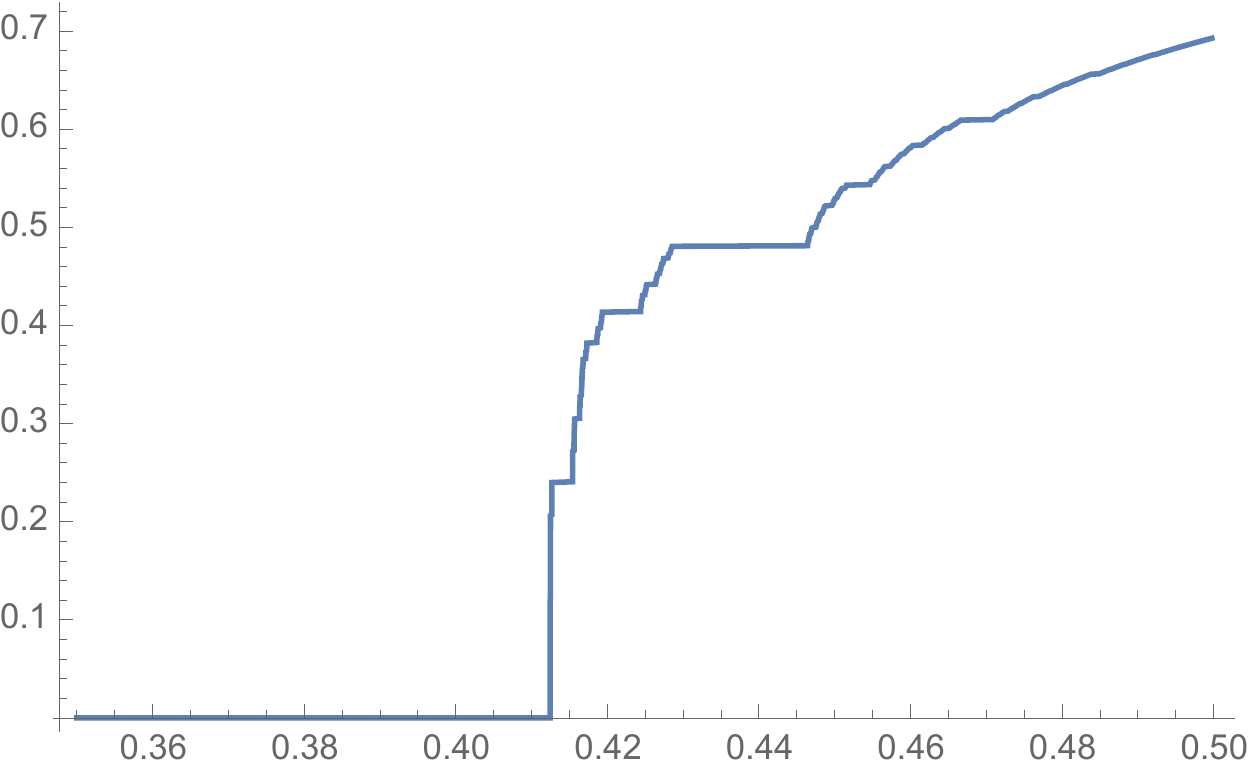}}
\caption{The entropy of real unimodal maps as a function of the external angle. Note the \emph{Feigenbaum angle} $\theta_\star  = 0.412454\dots$ where $h(\theta_\star) = 0$ and the function is continuous but not H\"older continuous. To the right of $\theta_\star$, the function is locally H\"older continuous, and becomes
more and more regular; it is actually almost Lipschitz continuous near the right endpoint $\theta = 1/2$.}
\end{figure}

\noindent This result confirms experimental evidence by Isola-Politi (\cite{IP}, page 282). 
%In our paper, we look at the dependence with respect to the combinatorial parameter, and use only combinatorial tools. 
The proof turns out to be a simple computation using the fact that entropy is the zero of a power series known as the \emph{kneading series}: 
basically, two nearby angles will produce power series with the same leading coefficients, and one needs to estimate how the zeros of a 
power series change as the coefficients change. In order to get the lower bound, however, one needs to 
analyze carefully the combinatorics of the set $\mathcal{R}$.

By a classical result of Guckenheimer \cite{Gu}, the topological entropy of any $C^1$ family of unimodal maps is a H\"older continuous function of the parameter.
More recently, a formula for the H\"older exponent of the entropy with respect to the analytic parameter has been obtained in \cite{DM}; 
in this paper, we look instead at the dependence on the combinatorial parameter.

%For the quadratic family $f_c(z) := z^2+c$, Guckenheimer \cite{Gu} proved that the topological entropy 
%of any $C^1$ family of unimodal maps 
%is a H\"older continuous function of the parameter $c$. 

A natural generalization of this discussion would be to extend the result to the \emph{core entropy} for the complex quadratic family, 
as defined by W. Thurston and studied in \cite{Ti}, \cite{DS}. 
In particular, it is proven in \cite{Ti} that the entropy $h(\theta)$ is locally H\"older continuous in a neighbourhood of rational external angles $\theta$ with $h(\theta) > 0$, 
and it is not locally H\"older where $h(\theta) = 0$.  %has been proven in \cite{Ti} and 
The exact value for the H\"older exponent has been conjectured independently by the author (see e.g. \cite{CT}, end of introduction) and \cite{BS}.
A proof of it has recently been announced by M. Fels (private communication).

\subsection*{Relation to open dynamical systems}

Theorem \ref{T:main} can be also reformulated in terms of open dynamical systems, and in this setting it is very close 
to the results (and the proof) of \cite{CT}. 

Let $D(x) := 2x \mod 1$ be the doubling map, and for each $\theta \in [0,1/2]$ let us consider
the set $K(\theta)$ of points in the circle whose forward orbit never intersects the interval $(\theta, 1-\theta)$. 
$$K(\theta) := \{ x \in \mathbb{R}/\mathbb{Z} \ : \ D^n(x) \notin (\theta, 1 - \theta) \quad \forall n \geq 0 \}$$
Then for each $\theta \in [0, 1/2]$ one has 
$$\textup{H.dim }K(\theta) = \frac{h(\theta)}{\log 2}.$$

\begin{corollary} \label{C:main}
Consider for each $\theta \in [0,1/2]$ the dimension function
$$\eta(\theta) := \textup{H.dim }K(\theta).$$ 
Then, for each $\theta$ not in a plateau, the local H\"older exponent of the dimension function satisfies 
$$\alpha(\eta, \theta) = \eta(\theta).$$
\end{corollary}

A completely analogous statement for the set of points whose orbit does not intersect the forbidden interval $(0,t)$ is proven in \cite{CT}. 
The continuity of entropy (or dimension) for expanding circle maps with holes has been established in the 1980s be Urba{\'n}ski \cite{Urb},
 while bounds on the H\"older exponent of the dimension for more general holes are proven in \cite{BR}. 

Note that the dimension function of \cite{CT} and the one considered in this paper are genuinely different functions: for example, we will check 
in Section \ref{S:Feig} that the modulus of continuity of $h$ near $\theta_\star$ is of order $\frac{1}{\log(1/x)}$, while for the dimension function 
considered in \cite{CT} the modulus of continuity at the point where it is not H\"older is of order $\frac{\log \log (1/x)}{\log(1/x)}$.

\subsection*{Acknowledgments}
This paper is dedicated to the memory of Tan Lei. It would be hard for me to overstate the support and encouragement I got from her, and 
I will be always grateful. Many of the ideas we discussed are still to be explored, and I very much hope they will be eventually worked out
with the help of the mathematical community.

\section{Background material} \label{S:back}

\subsection*{H\"older exponents} 
Let $I \subseteq \mathbb{R}$ be an interval.  A function $f : I \to \mathbb{R}$ is \emph{locally H\"older continuous} at a point $x$ of exponent $\alpha$ 
if there exists a neighborhood $U$ of $x$ and a constant $C > 0$ such that 
$$|f(y) - f(z)| \leq C |y-z|^\alpha \qquad \textup{for all } y, z \in U.$$
The \emph{local H\"older exponent} of $f$ at $z$ is 
$$\alpha(f, z) := \sup \left\{ \eta > 0 \ : \ \lim_{\epsilon \to 0} \sup_{\stackrel{|x-z| < \epsilon}{|y-z| < \epsilon}} \frac{|f(x) - f(y)|}{|x-y|^\eta} < + \infty  \right\}. $$
By definition, $f$ is locally H\"older continuous at $z$ if and only if $\alpha(f, z) > 0$. 

\subsection*{Kneading theory}
Let now $I = [0,1]$. Recall that a \emph{unimodal map} is a continuous function $f : I \to I$ with $f(0) = f(1) = 0$, and for which there exists a point $c \in (0,1)$, which we call \emph{critical point}, such that $f$ is increasing on $[0, c)$ and decreasing on $(c, 1]$. In order to capture the symbolic dynamics of 
$f$, one defines the \emph{address} of a point $x \neq c$ as 
$$A(x)  := \left\{ \begin{array}{ll} 0 & \textup{if } x < c \\
1 & \textup{if } x > c.\\
\end{array} \right.$$
The \emph{kneading sequence} of $f$ is then defined as the sequence of addresses of the iterates of the critical point; namely, 
for any $k \geq 1$ set, if $f^k(c) \neq c$, 
$$s_k := A(f^k(c))$$
%$$\left\{ \begin{array}{ll} 0 & \textup{if } f^k(c_0) < c_0 \\
%1 & \textup{if } f^k(c_0) > c_0\\
%\end{array} \right.$$
while, if $f^k(c) = c$, then set 
$$s_k := \lim_{x \to c} A(f^k(x))$$
which is still well-defined as $f$ ``folds" a neighbourhood of $c$. Finally, one defines the \emph{kneading angle} as 
$$\theta := \sum_{k = 1}^\infty \frac{\theta_k}{ 2^{k+1}} \qquad \textup{with}\qquad 
\theta_k := s_1 + \dots + s_k \mod 2.$$
%Let $\theta = \sum_{k = 0}^\infty \theta_k 2^{-k-1}$ be the binary expansion of $\theta$. 
Then the \emph{kneading series} associated to the angle $\theta$ is the power series
$$P_\theta(t) = 1 + \sum_{k = 1}^\infty \epsilon_{k} t^k$$
where $\epsilon_k = (-1)^{\theta_k}$. Note that the coefficients of $P_\theta(t)$ are uniformly bounded, hence the power series 
defines a holomorphic function in the unit disk $\{ t \in \mathbb{C} \ : \ |t| < 1\}$. The main result of kneading theory is the following:

\begin{theorem}[\cite{MT}]
Let $f$ be a unimodal map with kneading angle $\theta$ and topological entropy $h(f)$, and let $r := e^{-h(f)}$. Then the 
holomorphic function 
$P_\theta(t)$ is non-zero on the disk $\{ t \in \mathbb{C} \ : \ |t| < r\}$, and  if $r < 1$ one has 
$$P_\theta(r) = 0.$$
\end{theorem}

\subsection*{The set of real angles}
It is well-known that not all binary sequences are indeed kneading sequences of unimodal maps: 
let us denote as $\mathcal{R}$  the set of all possible kneading angles. This coincides (up to possibly a set of Hausdorff dimension zero) with the set of angles of 
external rays which land on the real slice of the Mandelbrot set. 

The angles corresponding to real parameters are characterized in terms of the dynamics of the angle doubling map 
$D(\theta) := 2\theta \mod 1$. In fact, from \cite{Do} one as the identity 
$$\mathcal{R} = \{ \theta \in [0, 1/2] \ :  \ D^n(\theta) \notin (\theta, 1- \theta) \ \textup{for all }n \geq 0\}.$$
(Obviously, the set of angles of external rays landing on the real slice of the Mandelbrot set is symmetric about $1/2$, but we will only focus 
on the interval $[0, 1/2]$ in this paper.)
Many facts are known about the structure of the set $\mathcal{R}$. In particular, it is a closed set of Lebesgue measure zero and 
Hausdorff dimension one \cite{Za}. Moreover, one knows that for each angle $\theta \in \mathcal{R}$ which is purely periodic for the 
doubling map, one can produce its \emph{period doubling} $\theta'$ as follows: if $\theta$ has (minimal) period $p$ and its binary expansion is 
$$\theta = .\overline{s_1 \dots s_p},$$ 
then we define 
$$\theta' = .\overline{s_1 \dots s_p \check{s}_1 \dots \check{s}_p }$$ 
where $\check{s}_i = 1-s_i$. The following lemma is well-known.

\begin{lemma} \label{L:pd}
If $\theta'$ is the period doubling of $\theta \in \mathcal{R}$, then $\theta' \in \mathcal{R}$, and moreover
$$h(\theta') = h(\theta).$$
\end{lemma}

\begin{proof}
It is immediate to check by the definitions that the two kneading series are related by
$$P_{\theta'}(t) = P_\theta(t) \frac{1- t^p}{1+t^p}$$
from which the claim holds, as $ \frac{1- t^p}{1+t^p}$ does not have any roots inside the unit disk.
\end{proof}

\noindent Moreover, the connected components of the complement of $\mathcal{R}$ are precisely
$$[0, 1/2] \setminus \mathcal{R} = \bigcup_{\theta \in \mathcal{R}^0} (\theta, \textup{pd}(\theta))$$
where $\textup{pd}(\theta)$ is the period doubling of $\theta$, and 
$\mathcal{R}^0$ is the subset of $\mathcal{R}$ consisting of purely periodic angles.
Note that $\theta = 0 = .\overline{0}$ belongs to $\mathcal{R}$, and the equation above is correct if one interprets its period doubling to be 
$\textup{pd}(0) = .\overline{01} = \frac{1}{3}$.

%One defines the \emph{kneading sequence} as $(\eta_k)_{k \geq 1} := \textup{sign}(f^k(c))$.
%$$\eta_k := \lim_{x \to c} \textup{sign}(f^k(x))$$
%$$\epsilon_k := \eta_1 \eta_2 \dots \eta_k \epsilon_0$$

For each purely periodic $\theta \in \mathcal{R}$, one defines the \emph{small copy} of root $\theta$ as the interval 
$$I(\theta) := (\theta, \overline{\theta})$$
where $\overline{\theta}$ has binary expansion $\overline{\theta} := .s_1\dots s_p \overline{\check{s}_1 \dots \check{s}_p}$.
The reason for the name ``small copy" is that the interval $I(\theta)$ corresponds to the set of external rays landing on the real slice of the small 
copy of the Mandelbrot set with root of external angle $\theta$. 

\begin{lemma}[\cite{Do}]
If $h(\theta) > 0$, then the entropy is constant on the small copy $I(\theta)$.
\end{lemma}

\section{Simplicity of the minimal root}

We start by proving that the smallest root of the kneading determinant is actually simple. This fact may be of independent interest, and is probably 
known to experts even though we could not find it in the literature.

\begin{theorem} \label{T:simple}
Let $f : I \to I$ be a unimodal map with topological entropy $h(f) > 0$.  Denote as $P(t)$ its kneading series, and let $s = e^{h(f)}$. Then $r = \frac{1}{s}$ is a simple root of $P(t)$.
\end{theorem}

\begin{proof}
Let us assume first that $f$ is piecewise linear with slope $\pm s$. Then, by \cite[Proposition 9.6]{Preston} $r$ is a simple pole of $L(t)$, hence it is a simple zero of $P(t)$.
Now, in the general case, by Milnor-Thurston \cite{MT}, for any unimodal map of entropy $h(f) = \log s$ there exists a semiconjugacy $\pi : I \to J$ of $f$ to a piecewise linear unimodal map $g : J \to J$ of slope $\pm s$. That is, there is a continuous map $\pi : I \to J$ such that $\pi \circ f = g \circ \pi$.
%Suppose $h(f) > \frac{\log 2}{2}$. 
Denote as $c$ the critical point of $f$, and let $\widetilde{c} = \pi(c)$ be the critical (or \emph{turning}) point of $g$. Moreover, let $L := \pi^{-1}(\widetilde{c})$, which is a closed interval containing $c$. There are two cases: 
\begin{enumerate}
\item
either $f^n(c) \notin L$ for all $n \geq 1$. This implies that 
$$P_f(t) = P_g(t)$$
hence the claim follows, since from above we know that $r$ is a simple root of $P_g(t)$.
\item 
Otherwise, there exists $n$ such that $f^n(c) \in L$. Let $p$ be the smallest such $n$. This implies that $g^p(\widetilde{c}) = g^p(\pi(c)) = \pi(f^p(c)) = \widetilde{c}$, since $f^p(c) \in L$. Then we get the factorisation
$$P_f(t) = \widetilde{P}_g(t) P_h(t^p)$$
where $h = f^p\mid_L$ is the first return map of $f$ to $L$, which is also a unimodal map, and $\widetilde{P}_g(t)$ is the polynomial (of degree $p-1$)
such that 
$$P_g(t) = \frac{\widetilde{P}_g(t)}{1-t^p}.$$
We now claim that $r$ is a root of 
$P_g(t)$ and \emph{not} a root of $P_h(t^p)$, which implies by the above factorisation that 
$r$ is a simple root of $P_f(t)$. 
%Now, by \cite[Proposition 9.6]{Preston} $r$ is a simple pole of $L_g(t) = \ell(g, t)$, hence it is a simple zero of $P_g(t)$.
%To prove the claim, note that if $g$ is a superattracting unimodal map with 
In order to prove the claim, let $k \geq 1$ be the integer such that 
%\begin{equation}\label{E:h-in}
$$\frac{\log 2}{2^k} < h(g) \leq \frac{\log 2}{2^{k-1}}.$$
%\end{equation}
Note that the above inequality implies that the period $p$ of $g$ must be at least $2^k$. 
%Thus, if we choose $k$ to be the largest integer $\geq 0$ such that \eqref{E:h-in} holds, 
%by the above remark we have $p \geq 2^k$. 
On the other hand, since $h$ is unimodal, each of the roots of $P_h(t)$ has modulus $\geq \frac{1}{2}$, thus each root of $P_h(t^p)$ has modulus at least $\frac{1}{2^{1/p}} \geq \frac{1}{2^{1/2^k}} > e^{-h(g)} = r$. This means that $r$ is not a root of $P_h(t^p)$, and by the above discussion $r$ is 
a simple root of $P_g(t)$. This proves that $r$ is a simple root of $P_f(t)$. 
%This means by the above discussion that $r$ is a simple root of $\widetilde{P}_g(t)$ and not a root of $P_h(t^p)$, hence by the above factorisation it is a simple root of $P_f(t)$.
\end{enumerate}
\end{proof}

%\begin{lemma}
%Let $g$ be a tent map of slope $s > 1$. Then $\lambda = 1/s$ is a simple root of $P_g(t)$.
%\end{lemma}

%\begin{proof}
%Note that $g^n$ is also piecewise linear with derivative $= \pm s^n$ at all points except finitely many. Then 
%$\textup{Var}(g^n) = s^n (b-a)$.
%Then 
%$$V(t) = \sum_{n = 0}^\infty \textup{Var}(g^n) t^n = \frac{b-a}{1-s t}$$
%Thus $t = \frac{1}{s}$ is a simple pole of $V(t)$. Thus $t = \frac{1}{s}$ is a simple root of $P(t)$.
%\end{proof}

\begin{lemma} \label{L:deriv-bound}
Let $\theta$ be a real angle with $h(\theta) >0$, and let $r = e^{-h(\theta)}$. Then there exists $\epsilon > 0$ such that for any $\theta'$ with $|\theta- \theta'| < \epsilon$ the 
kneading series $P_{\theta'}(t)$ has exactly one root (counted with multiplicity) inside the disk $\{z \in \mathbb{C} \ : |z-r|  < \epsilon\}$.
Moreover, there exists $c > 0$ such that 
$$|P'_{\theta'}(t)| \geq c$$
for all $|\theta' - \theta| < \epsilon$ and $|t-r| < \epsilon$.
\end{lemma}

\begin{proof}
There are two cases. If $\theta$ is not purely periodic for the doubling map, then for $\theta'$ sufficiently close to $\theta$ the coefficients of $P_{\theta'}(t)$ 
eventually stabilize to the coefficients of $P_\theta(t)$. Hence, $P_{\theta'}(t)$ converges to $P_{\theta}(t)$ uniformly on compact subsets of the unit disk.
Moreover, Theorem \ref{T:simple} implies that $r$ is the only root of $P_\theta(t)$ in a neighborhood of $z = r$, counting with multiplicities, so the 
first claim follows by Rouch\'e's theorem. On the other hand, if $\theta$ is purely periodic of period $p$, then there are two possible limits of the power series $P_{\theta'}(t)$ 
as $\theta' \to \theta$. Indeed, one checks that 
$$\lim_{\theta' \to \theta^-} P_{\theta'}(t) = P_\theta(t)$$
while 
$$\lim_{\theta' \to \theta^+} P_{\theta'}(t) = \widehat{P}_\theta(t) = P_\theta(t) \frac{1-t^p}{1+t^p}.$$
Since the function $\varphi(t) = \frac{1-t^p}{1+t^p}$ is never vanishing inside the unit disk, then the claim follows by Rouch\'e's theorem as before.
To prove the second claim, note that we just proved that $P'_{\theta}(r) \neq 0$, and $\widehat{P}'_\theta(r) \neq 0$, so the claim follows by noting that the derivative $P'_{\theta'}(t)$ also converges uniformly on compact sets to either $P'_{\theta}(t)$ or $\widehat{P}'_{\theta}(t)$.

\end{proof}

\section{The local H\"older exponent: the upper bound}

\begin{proposition}\label{P:upper}
For each $\theta \in \mathcal{R}$ with $h(\theta) > 0$, there exists $C = C(\theta) > 0$ such that the modulus of continuity of the entropy is bounded by   
$$|h(\theta) - h(\theta')| \leq C |\theta - \theta'|^{\frac{h(\theta)}{\log 2}}$$
for each $\theta, \theta' \in \mathcal{R}$ with $\theta' < \theta \leq 1/2$. 
\end{proposition}

%Recall that $\theta_0$ is the angle of the Feigenbaum parameter, i.e. the largest value of $\theta$ for which $h(\theta) = 0$.

\begin{proof}
Let $\theta, \theta'$ be two real angles, and let us assume $\theta' < \theta \leq \frac{1}{2}$.  Let their binary expansions be 
%parameters, and consider the kneading series $P_{\theta_1}(t)$ and $P_{\theta_2}(t)$. 
%Recall that 
%$$P_{\theta}(r) = 0$$
%Note that $P_\theta(r) = P_{\theta_1}(r_1) = 0$. 
%Consider two angles $\theta, \theta_1$, with binary expansions 
$\theta = \sum_{k = 1} \theta_k 2^{-k}$ and $\theta' = \sum_{k = 1} \theta'_k 2^{-k}$ and let $n := \min \{ k \ : \ \theta_k \neq \theta_k' \}$.
Then by Lemma \ref{L:dist} one gets 
$$c 2^{-n} \leq |\theta - \theta'| \leq 2^{-n+1}.$$
Let us now compare the two kneading series $P_\theta(t)$ and $P_{\theta'}(t)$. As the first $n-1$ coefficients of the two series coincide, one gets 
\begin{equation}\label{E:eq1}
P_{\theta}(t) -  P_{\theta'}(t) = 2 t^n + \sum_{k = n+1}^\infty (\epsilon_k - \epsilon'_k) t^k = t^n h(t)
\end{equation}
where $h(t) = 2 + \sum_{k = 1}^\infty (\epsilon_{n+k} - \epsilon'_{n+k}) t^k$. On the other hand, as $P_\theta(r) = P_{\theta'}(r') = 0$, one has
\begin{equation} \label{E:eq2}
P_{\theta}(r) - P_{\theta'}(r) = P_{\theta'}(r') - P_{\theta'}(r) = P'_{\theta'}(\xi)(r' - r)
\end{equation}
with $\xi \in [r, r']$. Thus, combining the two previous equations we get 
\begin{equation}\label{E:main}
r' - r = r^n \frac{h(r)}{P'_{\theta'}(\xi)}.
\end{equation}
% \frac{P_{\theta}(t_0) - P_{\theta}(t_1)}{P'_{\theta}(\xi)} = \frac{ 2 t^n + \sum (\epsilon_k - \epsilon'_k) t^k }{P'_{\theta_0}(\xi)}$$

In order to get the upper bound, let us note that $|h(r)| \leq \frac{2}{1-r}$ as the coefficients of $h(t)$ are bounded in absolute value by $2$.
Moreover, by Lemma \ref{L:deriv-bound} we have
$$\inf_{\stackrel{|\theta - \theta'| < \epsilon}{|\xi - \theta| < \epsilon}} |P'_{\theta'}(\xi)| = c_1 > 0.$$
Finally, by Lemma \ref{L:dist} one gets 
$$n \geq \frac{\log c - \log |\theta - \theta'|}{\log 2}$$
hence 
\begin{equation}
r^n = e^{n \log r} \leq c_2 |\theta - \theta'|^{\frac{- \log r}{\log 2}}
\end{equation}
where  $c_2 = e^{\frac{\log r \log c}{\log 2}}$. Thus, putting together the previous estimates 
\begin{equation}
r' - r \leq \frac{c_2}{c_1(1-r)}  |\theta - \theta'|^{\frac{- \log r}{\log 2}}
\end{equation}
which using the definition $h(\theta) = - \log r$  yields the upper bound 
$$r' - r \leq C  |\theta - \theta'|^{\frac{h(\theta)}{\log 2}}$$
where we set $C = \frac{c_2}{c_1(1-r)}$. The claim then follows as $h(\theta) = - \log r$ and the function $x \mapsto \log x$ is 
differentiable with bounded derivative (hence Lipschitz) on the interval $[1, 2]$.

%In order to prove the lower bound, let $\theta$ be purely periodic of period $p$. Then $\theta$ is represented by an extremal word $S$. By Lemma \ref{}, 
%there exists a dominant word $T$ such that $S << T$. Then, by Lemma \ref{} the word 
%$$S^m T^n$$ 
%is extremal for all $m, n \geq 1$. Let $p = |S|$ and $q = |T|$, and $P = pq$. Then consider the word 
%$$S_m := S^{mq} T^p$$
%and let $\theta_m$ be its associated binary word. 

%Then for any $m \geq 1$ one gets the identity 
%$$P_\theta(t) - P_{\theta_m}(t) = \frac{g(t)t^{P m}}{1 - t^{P(m+1)}}$$
%where $g(t) = \sum_{k = P +1}^{2 P} (\epsilon_k - \epsilon'_k) t^k$.
%Hence, by combining with equation \ref{}
%$$r_m - r = \frac{r^{Pm} g(r)}{(1 - r^{P(m+1)}) P_{\theta_m}'(\xi_m)}$$
%with $\xi_m \in [r, r_1]$.
%Now, note that 
%$1 - r^{P(m+1)} \leq 1$ and $|P'_{\theta_m}(\xi_m)| \leq C$, hence 
%$$| r_m - r | \geq \frac{r^{Pm} g(r)}{C}$$
%and $\theta - \theta_m \leq 2^{-m P}$
\end{proof}

\begin{lemma} \label{L:dist}
Let $\theta, \theta' \in \mathcal{R}$ with $0 < \theta' < \theta \leq \frac{1}{2}$, and $n = \min \{ k \ : \ \theta_k \neq \theta_k' \}$. Then 
$$c 2^{-n} \leq |\theta - \theta'| \leq 2^{-n+1}$$
where $c = 2(1 - 2 \theta)$ if $\theta < 1/2$, and $c = 1$ if $\theta = 1/2$. 
\end{lemma}

\begin{proof}
Let us first consider the case $\theta < 1/2$. Recall that the set $\mathcal{R}$ is characterized as 
$$\mathcal{R} = \{ \theta \in [0, 1/2] \ : \ D^n(\theta) \notin (\theta, 1 - \theta) \quad  \forall n \geq 0\}.$$
Now, by definition of $n$ one has
$D^{n-1}(\theta') < \frac{1}{2} < D^{n-1}(\theta)$
hence by the definition of $\mathcal{R}$
$$D^{n-1}(\theta') \leq \theta' < \theta < \frac{1}{2} < 1 - \theta \leq D^{n-1}(\theta).$$
This implies
$$2^{n-1}(\theta - \theta') = D^{n-1}(\theta) - D^{n-1}(\theta') \geq 1 - 2 \theta$$
which yields the lower bound. 

If $\theta = 1/2$, then $D^{n-1}(\theta) = 1$, hence $D^{n-1}(\theta) - D^{n-1}(\theta') \geq 1/2$, and 
the proof proceeds as before.

The upper bound follows simply because 
$$2^{n-1}(\theta - \theta') = D^{n-1}(\theta) - D^{n-1}(\theta') \leq 1.$$
\end{proof}

\section{Primitive angles}

In order to prove the lower bound for the local H\"older exponent we need the following definition.

\begin{definition}
An angle $\theta \in \mathcal{R}$ is called \emph{primitive} if it is purely periodic for the doubling map, and moreover such that $D^k(\theta) \neq 1 - \theta$ for all $k \geq 0$.
\end{definition}

A purely periodic, real angle which is not primitive will be called \emph{satellite}. The external rays corresponding to these parameters 
land at roots of satellite components of the Mandelbrot set. 

\begin{lemma} \label{L:satellite}
If $\theta \in \mathcal{R}$ is satellite, then the entropy is locally constant at $\theta$. 
\end{lemma}

\begin{proof}
Indeed, if $D^k(\theta) = 1-\theta$ then the binary expansion of $\theta$ is of the form 
$$\theta = .\overline{ s_1 \dots s_k \check{s}_1 \dots \check{s}_k}$$
where $\check{s}_i := 1 - s_i$. This means that $\theta$ is the period doubling of the angle $\theta' = .\overline{s_1 \dots s_k}$, thus $h(\theta) = h(\theta')$ by Lemma \ref{L:pd}.
\end{proof}

The reason we introduce this definition is because we need it to prove that primitive angles can be approximated by real angles with controlled combinatorics.

\begin{lemma}
Let $\theta \in \mathcal{R} \cap (0, \frac{1}{2})$ be a primitive angle with $D^p(\theta) = \theta$.
% and $D^k(\theta) \neq 1  - \theta$ for all $k$. 
Pick $\delta > 0$ such that $D^k(\theta) \notin [\theta, 1 - \theta  + 2\delta]$ for $0 < k < p$, and let $\theta' \in \mathcal{R}$ be a purely periodic angle with 
$\theta-\delta < \theta' < \theta$ and $D^q(\theta') = \theta'$. Let 
$$\theta = \sum_{k = 1}^\infty s_k 2^{-k} \qquad \textup{and }\qquad  \theta' = \sum_{k = 1}^\infty t_k 2^{-k}$$
%.\overline{s_1 \dots s_p}$ and $\theta' = .\overline{t_1 \dots t_q}$ 
be the binary expansions of $\theta$ and $\theta'$, respectively. Then the point of binary expansion
$$\xi := .\overline{s_1 \dots s_p t_1 \dots t_q}$$
belongs to $\mathcal{R}$.
%Let $\theta = .\overline{s}$ and $\theta' = .\overline{t}$.
%Then there exists $\theta' < \theta$ such that for all $m, n$, the angle 
%$$\theta_{m,n} := .\overline{s^m t^n}$$
%belongs to $\mathcal{R}$.
\end{lemma}

\begin{proof}
Since the map $D^{p}$ is uniformly expanding, if we let $x = .s\overline{t} \in [\theta', \theta]$, then $D^{p}$ is a homeomorphism 
between $[x, \theta]$ and $[\theta', \theta]$. Similarly, the point $y = .t \overline{s} \in [\theta', \theta]$ is so that $D^{q}$ is a homeomorphism between $[\theta', y]$ and $[\theta', \theta]$. 
%Let us pick now $z \in [x, \theta]$ such that $D^{p}(z) = y$, so that $D^{p}([x, z]) = [\theta', y]$ and $D^{p+q}([x,z]) = [\theta', \theta]$. 
We have that $\xi \in [x, \theta]$ and $D^p(\xi) \in [\theta', y]$, with $D^{p+q}(\xi) = \xi$. 
We now check that $\xi$ belongs to $\mathcal{R}$. 
%First of all, the map $D^{mp}$ is an expanding homeomorphism on $[x, \theta]$ and fixes $\theta$
%$$0 \leq D^k(\theta) - D^k(\theta_{m,n}) \leq \theta - \theta'\qquad \textup{for all }0 \leq k \leq mp$$
%Moreover, $D^{nq} : [\theta', y] \to [\theta', \theta]$ is a homeomorphism and $D^{mp}(\theta_{m,n})$ belongs to $[\theta', y]$, then 
%$$0 \leq D^{mp + k}(\theta_{m, n}) - D^{mp + k}(\theta') \leq \theta - \theta' \qquad \textup{for all }0 \leq k \leq nq$$
%Now, let us choose $\delta > 0$ such that for each $k$
%$$D^k(\theta) \in [0, \theta] \cup [1-\theta+\delta, 1]$$
%and pick $\theta' \in \mathcal{R} \cap [\theta, \theta- \delta]$.
%Then, if $D^k(\theta) \leq \theta$, then 
%$$D^k(\theta) - D^k(\theta_{m,n}) \geq \theta - \theta_{m,n}$$
%hence if $D^k(\theta) \leq \theta$ then 
%$$D^k(\theta_{m,n}) \leq \theta_{m,n} + D^k(\theta) - \theta \leq \theta_{m,n}$$
%on the other hand, when $D^k(\theta) \geq 1 - \theta + \delta$ then 
%$$D^k(\theta_{m,n}) \geq D^k(\theta) - \delta$$
%\end{proof}
%\begin{proof} 
In order to do so, we will check that for all iterates $0 < k < p+q$ the point $D^k(\xi)$ does not lie in the ``forbidden" interval $(\xi, 1 - \xi)$.
Let us first consider the earlier iterates, $D^k(\xi)$ with $0 < k \leq p$. Then as $D^k$ is expanding and orientation preserving on 
$[x, \theta]$ one gets the estimates
$$\theta - \xi \leq D^k(\theta) - D^k(\xi) \leq D^p(\theta) - D^p(\xi) \leq \theta - \theta'. $$
There are two cases: 
\begin{itemize}
\item 
If $D^k(\theta) < \frac{1}{2}$, then $D^k(\theta) \leq \theta$, hence 
$$D^k(\xi) \leq (D^k(\theta) - \theta) + \xi \leq \xi$$
\item
If $D^k(\theta) > \frac{1}{2}$, then by hypothesis $D^k(\theta) \geq 1- \theta + 2 \delta \geq 1- \theta + 2 (\theta - \theta')$, which implies
$$D^k(\xi) \geq D^k(\theta) - \theta + \theta' \geq 1 - \theta' \geq 1 - \xi$$
as required.
\end{itemize}
Let us consider now $D^{p+k}(\xi)$, with $0 < k < q$. Recall that by construction $D^{p+q}(\xi) = \xi$. Moreover, 
the map $D^k$ is expanding and orientation-preserving on $[\theta', y]$, thus 
$$0 \leq D^{p+k}(\xi) - D^k(\theta') \leq  D^{p+q}(\xi) - D^q(\theta') = \xi - \theta'.$$
Now, there are two cases: 
\begin{itemize}
\item Suppose $D^k(\theta') < \frac{1}{2}$. Then $D^k(\theta')\leq \theta'$ as $\theta'$ belongs to $\mathcal{R}$, 
hence one gets
%$$D^{k+p}(\xi) - D^k(\theta') \leq D^{p+q}(\xi) - D^q(\theta') = \xi - \theta'$$ 
%hence since $D^k(\theta')\leq \theta'$ as $\theta'$ belongs to $\mathcal{R}$ one gets
$$D^{p+k}(\xi) \leq \xi + (D^k(\theta') - \theta') \leq \xi$$
as required.
\item Suppose instead $D^k(\theta') > \frac{1}{2}$. Then $D^k(\theta') \geq 1 - \theta'$ hence we can write
$$D^{p+k}(\xi) \geq D^k(\theta') \geq 1 - \theta' \geq 1 - \xi.$$
\end{itemize}
In both cases, $D^{p+k}(\xi)$ belongs to $[\xi, 1 - \xi]$, hence $\xi$ belongs to $\mathcal{R}$.
\end{proof}

\begin{corollary} \label{C:many}
In the hypotheses of the previous lemma, for each $m$ the sequence of points 
$$\theta_m = .\overline{s^m t}$$
belongs to $\mathcal{R}$.
\end{corollary}

\begin{corollary} \label{C:approx}
If $\theta \in \mathcal{R} \cap (0, 1/2)$ is a purely periodic, primitive angle, then for any $\delta > 0$ there exists a purely periodic $\theta' \in \mathcal{R}$ with $\theta - \delta < \theta' < \theta$.
\end{corollary}

\section{The lower bound}

\begin{proposition} \label{P:lower}
Let $\theta \in \mathcal{R}$ be a primitive angle which does not lie in a plateau. Then one has the lower bound
%There exists a sequence $\theta_m$ of real angles with $\theta_m \to \theta$ as $m \to \infty$ and such that 
$$\liminf_{\theta' \to \theta} \frac{|h(\theta) - h(\theta')|}{|\theta - \theta'|^{\frac{h(\theta)}{\log 2}}} = c > 0.$$
\end{proposition}

\begin{proof}
Let $\theta \in \mathcal{R} \cap (0, \frac{1}{2})$ be a purely periodic, primitive angle of period $p$, and let us define
%, and such that $D^k(\theta) \neq 1 - \theta$ for all $k$. 
$$\delta := \frac{1}{2} \min \left\{ D^k(\theta) - 1 + \theta \ : \ k \geq 0, \ D^k(\theta) > 1- \theta \right\} >0.$$ 
Then, by Corollary \ref{C:approx}, there exists a purely periodic $\theta' \in \mathcal{R} \cap (\theta - \delta, \theta)$: let $q$ be the period of $\theta'$, and denote $P = pq$. Now by Corollary 
\ref{C:many}, for any $m$, the angle
$$\theta_m = .\overline{(\epsilon_1\dots \epsilon_{P})^m (\epsilon'_1 \dots \epsilon'_P)}$$
belongs to $\mathcal{R}$, where $(\epsilon_k)$ and $(\epsilon'_k)$ are, respectively, the digits in the binary expansions of $\theta$ and $\theta'$.
Now, if we let $g(t) = \sum_{k = 1}^{P} (\epsilon_k - \epsilon'_k) t^{k-1}$ then for each $m$ the difference between the two kneading series $P_{\theta_m}(t)$ and $P_{\theta}(t)$ can be written as 
$$P_{\theta}(t) - P_{\theta_m}(t) = %\sum_{n = 0}^\infty \sum_{k = 1}^{P}(\epsilon_k - \epsilon'_k) t^{mP + k-1 + nP(m+1)} = 
\frac{  \sum_{k = 1}^{P}(\epsilon_k - \epsilon'_k) t^{mP + k-1} }{1 - t^{P(m+1)}} = \frac{g(t) t^{mP}}{1 - t^{P(m+1)}}.$$
Thus, by denoting as $r_m$ the smallest real root of $P_{\theta_m}$ and using $P_\theta(r) = P_{\theta_m}(r_m) = 0$, one has
$$P_{\theta}(r) - P_{\theta_m}(r) = P_{\theta_m}(r_m) - P_{\theta_m}(r) = P_{\theta_m}'(\xi)(r_m - r) $$
for some $\xi \in [r, r_m]$. By combining the previous equations we get 
\begin{equation}\label{E:main-m}
r_m - r = \frac{1}{P_{\theta_m}'(\xi)}\frac{g(r) r^{mP}}{1 - r^{P(m+1)}}.
\end{equation}
Note that, since $\theta$ does not lie in a plateau, then $r_m \neq r$ for $m$ sufficiently large, hence $g(r) \neq 0$. 
Now, observe that the binary expansions of $\theta_m$ and $\theta$ have  at least $mP$ common initial digits, hence 
$$|\theta - \theta_m| \leq 2^{-mP}$$
thus
$$r^{mP} = e^{mP \log r} \geq |\theta- \theta_m|^{-\frac{\log r}{\log 2}} = |\theta - \theta_m|^{\frac{h(\theta)}{\log 2}}.$$
On the other hand, the coefficient of $t^k$ in $P'_{\theta_m}(t)$ has modulus $\leq k+1$, hence 
$$|P'_{\theta_m}(t)| \leq \sum_{k = 0} (k+1) t^k = \frac{1}{(1 - t)^2}$$
so using $\xi < r_1 < 1$ one gets $\frac{1}{|P'_{\theta_m}(\xi)|} \geq (1 - r_1)^2$ hence setting $c_3 = g(r) (1 - r_1)^2$ yields the final estimate
\begin{equation}\label{E:final-lower}
r_m - r \geq c_3 |\theta - \theta_m|^{\frac{h(\theta)}{\log 2}}
\end{equation}
which as $\theta_m \to \theta$ for $m \to \infty$ establishes the required lower bound.
\end{proof}

\begin{proof}[Proof of Theorem \ref{T:main}]
The first claim follows directly from Proposition \ref{P:upper}. The second claim is proven by noticing that every $\theta \in \mathcal{R}$ 
which does not lie in a plateau can be approximated by primitive angles, hence the lower bound on the modulus 
of continuity follows directly from Proposition \ref{P:lower}.
\end{proof}

\subsection{The Feigenbaum point} \label{S:Feig} 
Let us recall in the end that the entropy function is not H\"older continuous at $\theta = \theta_\star$, and in fact one can compute its modulus of continuity
using the combinatorics of period doubling. 

Let us consider the binary string $(S_n)$ defined recursively as $S_0 := 0$ and $S_{n+1} := S_n \check{S}_n$. The limit $S_\infty := \lim_{n \to \infty} S_n$ 
is the well-known \emph{Thue-Morse} sequence, which is the binary expansion of the Feigenbaum angle $\theta_\star$.
For each $n$, the angle $\eta_n := .\overline{S_n}$ lands at the root of a hyperbolic component of period $2^n$ which is given by $n$-times period doubling of the main cardioid, and there is an associated small copy $M_n$ of the Mandelbrot set. We will consider the angle $\theta_n := .S_n \overline{\check{S}_n}$ whose ray lands at the tip of $M_n$, and so that $\theta_\star = \lim_{n \to \infty} \theta_n$. 
Since $\theta_n$ is given by tuning of the tip of the Mandelbrot set with a zero entropy map of period $2^n$, 
%Let us denote as $\theta_n := .S_n \overline{\check{S}_n}$ the external angle of the tip of the small Mandelbrot set with root of period $2^n$
%given by applying $n$ times the period doubling transformation. Then 
one has 
$$h(\theta_n) = \frac{\log 2}{2^n}$$
while by looking at the binary expansions one gets $\theta_n - \theta_\star \asymp 2^{-2^n}$, which yields the estimate
$$|h(\theta_n) - h(\theta_\star)| \asymp \frac{1}{- \log |\theta_n - \theta_\star| }$$
(up to a multiplicative constant) thus the modulus of continuity is of order $\frac{1}{\log (\frac{1}{x} )}$.

\subsection{Relation to open dynamical systems}

\begin{proof}[Proof of Corollary \ref{C:main}]
Let $f$ be a real quadratic polynomial of kneading angle $\theta$, and let $J$ be the Julia set of $f$, which we know to be locally connected. 
Thus, there exists a continuous Caratheodory map $\gamma : \mathbb{R}/\mathbb{Z} \to J$ which semiconjugates the doubling map $D$ to $f$.
Let us consider the set $\widetilde{K} := \gamma^{-1}(J \cap \mathbb{R})$ of external angles of rays which land on the real section of the Julia set. 
By \cite{Do}, this set can be characterized as 
$$\widetilde{K} = \{ x \in \mathbb{R}/\mathbb{Z} \ : \ D^n(x) \notin (\theta, 1-\theta) \quad \forall n \geq 1\}$$
Since the map $\gamma$ is finite-to-one, one gets the equality 
$$h(f) = h(f \mid_{J \cap \mathbb{R}}) = h(D \mid_{\widetilde{K}})$$
Moreover, the map $D$ is uniformly expanding with derivative $2$, hence 
$$\textup{H.dim }\widetilde{K} = \frac{h(D \mid_{\widetilde{K}})}{\log 2}$$
Finally, if we denote
$$K(\theta) :=  \{ x \in \mathbb{R}/\mathbb{Z} \ : \ D^n(x) \notin (\theta, 1-\theta) \quad \forall n \geq 0\}$$
then clearly 
$$D(\widetilde{K}) \subseteq K(\theta) \subseteq \widetilde{K}$$
hence 
$$\textup{H.dim }K(\theta) = \textup{H.dim }\widetilde{K} = \frac{h(f)}{\log 2} = \frac{h(\theta)}{\log 2}$$
thus the Corollary follows directly from Theorem \ref{T:main}.
\end{proof}

\medskip

\noindent \textsc{University of Toronto}

\noindent \email{tiozzo@math.utoronto.ca}


\begin{thebibliography}{99}

\bibitem[BR]{BR}
\textsc{O. F. Bandtlow, H. H. Rugh},
\textit{Entropy-continuity for interval maps with holes}, 
available at {\tt arXiv:1510.06043}.

\bibitem[BS]{BS}
\textsc{H. Bruin, D. Schleicher,} 
\textit{Hausdorff dimension of biaccessible angles for quadratic polynomials}, 
available at {\tt arXiv:1205.2544v2 [math.DS]}.

\bibitem[CT]{CT} 
{\sc C. Carminati, G. Tiozzo}, 
{\it The local H\"older exponent for the dimension of invariant subsets of the circle}, 
to appear, Ergodic Theory and Dynamical Systems, 2015. 

\bibitem[DM]{DM}
{\sc N. Dobbs, N. Mihalache}, 
{\it Diabolical entropy}, 
available at {\tt arXiv:1610.02563}.

\bibitem[Do]{Do}
{\sc A. Douady}, 
{\it Topological entropy of unimodal maps: monotonicity for quadratic polynomials},
in 
{\it Real and complex dynamical systems ({H}iller\o d, 1993)},
NATO Adv. Sci. Inst. Ser. C Math. Phys. Sci.
{\bf 464}, 65--87, Kluwer, Dordrecht, 1995.

\bibitem[DS]{DS}
\textsc{D. Dudko, D. Schleicher,} 
\textit{Core entropy of quadratic polynomials}, available at {\tt arXiv:1412.8760v1 [math.DS]}. With an appendix by W. Jung.


\bibitem[Gu]{Gu}
\textsc{J. Guckenheimer}, 
\textit{The growth of topological entropy for one dimensional maps}, 
in 
\textit{Global theory of dynamical systems: proceedings of an international conference held at Northwestern University, Evanston, Illinois, June 18-22, 1979}, 
edited by Z. Nitecki and C. Robinson, 
Lecture Notes in Math. {\bf 819}, 216-223, Springer, Berlin, 1980.

\bibitem[IP]{IP}
\textsc{S. Isola, A. Politi}, 
\textit{Universal Encoding for Unimodal Maps}, 
J. Stat. Phys. {\bf 61} (1990), 263--291.

\bibitem[Mi]{Mi}
{\sc J. Milnor}, 
{\it Periodic orbits, external rays and the Mandelbrot set: 
an expository account}, 
Ast\'erisque {\bf 261} (2000), 277-333.

\bibitem[MT]{MT}
\textsc{J. Milnor, W. Thurston,}
\textit{On iterated maps of the interval},
in \textit{Dynamical systems (College Park, MD, 1986--87)},
Lecture Notes in Math. \textbf{1342}, 465--563, Springer, Berlin, 1988.

\bibitem[Pr]{Preston}
\textsc{C. Preston}, 
\textit{What you need to know to knead}, 
Adv. Math. {\bf 78} (1989), 192-252.

\bibitem[Ti]{Ti} 
\textsc{G. Tiozzo,} 
\emph{Continuity of core entropy of quadratic polynomials}, Invent. Math. \textbf{203} (2016), no. 3, 891-921.

\bibitem[Ur]{Urb}
\textsc{M. Urba{\'n}ski},
\textit{On {H}ausdorff dimension of invariant sets for expanding maps
              of a circle},
Ergodic Theory Dynam. Systems {\bf 6} (1986), no. 2, 295--309.
 
\bibitem[Za]{Za}
\textsc{S. Zakeri},
\textit{External Rays and the Real Slice of the Mandelbrot Set},
Ergod. Th.  Dyn. Sys. \textbf{23} (2003), 637--660.

\end{thebibliography}
\end{document}